\documentclass{article}
\usepackage[utf8]{inputenc}
\usepackage{amsfonts,amsmath,amssymb,amsthm,graphics,subfig,enumerate,breqn,color, url}

\newtheorem{Theorem}{Theorem}

\newtheorem{Corollary}{Corollary}

\begin{document}
\title{Co-Prime Order graph of a finite abelian Group and Dihedral Group}
\author{Amit Sehgal*$^{a}$, Manjeet$^{b}$,Dalip Singh$^{c}$ \\ 
$^{a,b}$Department of Mathematics, \\
Pt. NRS Govt. College,Rohtak (Haryana), India  \\ 
$^{c}$Department of Mathematics, \\
Maharshi Dayanand University,Rohtak (Haryana), India  \\ 
$^a$amit\_sehgal\_iit@yahoo.com, $^b$sainimanjeet1994@gmail.com\\
$^c$dsmdur@gmail.com  
}

\maketitle
\begin{abstract}
The \textbf{Co-Prime Order Graph} $\Theta (G)$ of a given finite group is a simple undirected graph whose vertex set is the group $G$ itself, and  any two vertexes x,y in $\Theta (G)$ are adjacent if and only if $gcd(o(x),o(y))=1$ or prime. In this paper, we find a precise formula to count the degree of a vertex in the Co-Prime Order graph of a finite abelian group or Dihedral group $D_n$.We also investigate the Laplacian spectrum of the Co-Prime Order Graph $\Theta (G)$ when G is finite abelian p-group, ${\mathbb{Z}_p}^t \times {\mathbb{Z}_q}^s$ or Dihedral group $D_{p^n}$.



\end{abstract}

\section{Introduction}
It is common to generate graphs from groups. In \cite{Zel75}, the author studied $Intersection~Graphs$ defined on a finite Abelian Group.The $Cayley~digraph$ is also an important class of directed graphs defined on finite groups and readers may refer to \cite{Bud85}. Kelarev and Quinn \cite{kela00} introduced the directed power graph of a group as a directed graph whose vertex set is the group, and there is an arc from vertex $u$ to the other vertex $v$ whenever $v$ is a power of $u$. Motivated by this concept, Chakrabarty et al. \cite{chak09} defined the power graph $\mathcal{G}(G)$ of a group $G$ as a graph with $G$ as its vertex set, and there is an edge between two distinct vertices if one is a power of the other. Subarsha Banerjee \cite{Sub19} introduced the Co-Prime Order Graph of a group as a simple graph whose vertex set is the group, and there is an arc from vertex $u$ to the other vertex $v$ whenever $gcd(o(u),o(v))=1$ or prime. The degree of a vertex of a graph associated with a finite group is also used to study the structural properties of the graph \cite{Seh19}. In Section 3 of the present paper, we obtain the degree of a vertex in the Co-Prime Order graph of a finite abelian group or Dihedral group $D_n$. In Section 4 of the present paper, we obtain the Laplacian spectrum of the Co-Prime Order Graph $\Theta (G)$ when G is finite abelian p-group, ${\mathbb{Z}_p}^t \times {\mathbb{Z}_q}^s$ or Dihedral group $D_{p^n}$. 

\section{Preliminaries and Notation}
Let $\Gamma$ be a finite simple graph. Its Laplacian  matrix is the matrix $L(\Gamma) =D(\Gamma)~-~ A(\Gamma)$, where $D(\Gamma)$ is the diagonal matrix of vertex degrees of $\Gamma$ and $A(\Gamma)$is the adjacency matrix of $\Gamma$.The Laplacian polynomial of $\Gamma$ is the characteristic polynomial of $L(\Gamma)$

Let $G$ be a group. Let $e$ denote the identity element of $G$ and $|G|$ denote the order of $G$ used throughout the paper.  The cyclic group of order $n$ is usually denoted by $\mathbb{Z}_n$. Let $g$ be an arbitrary element of $G$. We denote the order of $g$ by $|g|$. Let $H$ and $K$ be two normal subgroups of $G$. 
\section{Results for Degree}

\begin{Theorem}
Let $G=G_1 \times G_2 \times \ldots \times G_r$ be a finite abelian group where $G_i \thickapprox \mathbb{Z}_{{p_i}^{{i_1}}} \times \mathbb{Z}_{{p_i}^{{i_2}}} \times \ldots \times \mathbb{Z}_{{p_i}^{{i_{ n_i}}}}$ where $p_i$ are distinct prime when $1 \leq i \leq r$, then we get following results for degree of various vertices in co-prime order graph of group G\\
(i) $\deg(x)=\prod_{k=1}^{r} {p_k}^{\sum_{j=1}^{{n_k}} k_j}-1$ when $o(x)=1 ~or~p_i$ where $i=1,2,\ldots,r$\\
(ii) $\deg(x)=(-k-l+1+\sum_{j=1}^{j=k} |{G_{{\alpha_j}}}|+\sum_{j=1}^{j=l} {({p_{\beta_j}})^{n_{\beta_j}}})\frac{|G|}{(\prod_{i=1}^{i=k}|G_{{\alpha_i}}|) (\prod_{i=1}^{i=l} |G_{{\beta_i}}|)}$ when $o(x)=(\prod_{i=1}^{i=k} p_{\alpha_i}) (\prod_{i=1}^{i=l} {p_{\beta_i}}^{\gamma_i})$ where $2\leq k+l \leq r$ or $l \geq 1$ and $p_{\alpha_1}, p_{\alpha_2},\ldots, p_{\alpha_k},p_{\beta_1},p_{\beta_2},\ldots,p_{\beta_l}$ are distinct primes from set $\{p_1,p_2,\ldots,p_r\}$ and $\gamma_1,\gamma_2,\ldots,\gamma_l \geq 2$.
\end{Theorem}
\begin{proof}
Let $x$ be an arbitrary element group G of order $\prod_{i=1}^{r} {p_i}^{\alpha_i}$, then there exist unique $x_{p_i} \in G_i$ of order ${p_i}^{\alpha_i}$ such that $x=\prod_{i=1}^{r} x_{p_i}$.
We also know that group $G_i$ has exactly ${p_i}^{n_i}-1$ elements of order $p_i$.\\
We consider following case as follows:\\
{\bf{Case 1:}} Let $o(x)=1~or~p_i$\\
We know that $gcd(o(x),o(y))=1~or~p_i$ for every $y \in G$,so vertex $x$ is connected all the vertices,hence $\deg(x)=\prod_{k=1}^{r} {p_k}^{\sum_{j=1}^{{n_k}} k_j}-1$.\\

{\bf {Case 2:}} Let $o(x)=(\prod_{i=1}^{i=k} p_{\alpha_i}) (\prod_{i=1}^{i=l} {p_{\beta_i}}^{\gamma_i})$ where $2\leq k+l \leq r$ or $l \geq 1$ and $p_{\alpha_1}, p_{\alpha_2},\ldots, p_{\alpha_k},p_{\beta_1},p_{\beta_2},\ldots,p_{\beta_l}$ are distinct primes from set $\{p_1,p_2,\ldots,p_r\}$ and $\gamma_1,\gamma_2,\ldots,\gamma_l \geq 2$.\\
We have $((\prod_{i=1}^{i=k} p_{\alpha_i}) (\prod_{i=1}^{i=l} {p_{\beta_i}}^{\gamma_i}),\frac{|G|}{(\prod_{i=1}^{i=k}|G_{{\alpha_i}}|) (\prod_{i=1}^{i=l} G_{{\beta_i}})})=1$,then\\
$((\prod_{i=1}^{i=k} p_{\alpha_i}) (\prod_{i=1}^{i=l} {p_{\beta_i}}^{\gamma_i}),\frac{|{G_{{\alpha_j}}}||G|}{(\prod_{i=1}^{i=k}|G_{{\alpha_i}}|) (\prod_{i=1}^{i=l} G_{{\beta_i}})})=p_{\alpha_j}$ where $1 \leq j \leq k$ and $((\prod_{i=1}^{i=k} p_{\alpha_i}) (\prod_{i=1}^{i=l} {p_{\beta_i}}^{\gamma_i}) ,\frac{{{p_{\beta_s}}}|G|}{(\prod_{i=1}^{i=k}|G_{{\alpha_i}}|) (\prod_{i=1}^{i=l} |G_{{\beta_i}}|)})=p_{\beta_s}$ where $1 \leq s \leq l$.

So, vertex $x$ is adjacent to every vertex whose order divides $\frac{|{G_{{\alpha_j}}}||G|}{(\prod_{i=1}^{i=k}|G_{{\alpha_i}}|) (\prod_{i=1}^{i=l} |G_{{\beta_i}}|)}$ or $\frac{{{p_{\beta_s}}}|G|}{(\prod_{i=1}^{i=k}|G_{{\alpha_i}}|) (\prod_{i=1}^{i=l} G_{{\beta_i}})})$  for every $1 \leq j \leq k$ and $1 \leq s \leq l$.

So, vertex $x$ is adjacent to $\sum_{j=1}^{j=k} (\frac{|{G_{{\alpha_j}}}||G|}{(\prod_{i=1}^{i=k}|G_{{\alpha_i}}|) (\prod_{i=1}^{i=l} |G_{{\beta_i}}|)}-\frac{|G|}{(\prod_{i=1}^{i=k}|G_{{\alpha_i}}|) (\prod_{i=1}^{i=l} |G_{{\beta_i}}|)})+\sum_{j=1}^{j=l} \frac{{({p_{\beta_j}})^{n_{\beta_j}}}|G|}{(\prod_{i=1}^{i=k}|G_{{\alpha_i}}|) (\prod_{i=1}^{i=l} |G_{{\beta_i}}|)})-\frac{|G|}{(\prod_{i=1}^{i=k}|G_{{\alpha_i}}|) (\prod_{i=1}^{i=l} |G_{{\beta_i}}|)})+\frac{|G|}{(\prod_{i=1}^{i=k}|G_{{\alpha_i}}|) (\prod_{i=1}^{i=l} |G_{{\beta_i}}|)}=(-k-l+1+\sum_{j=1}^{j=k} |{G_{{\alpha_j}}}|+\sum_{j=1}^{j=l} {({p_{\beta_j}})^{n_{\beta_j}}})\frac{|G|}{(\prod_{i=1}^{i=k}|G_{{\alpha_i}}|) (\prod_{i=1}^{i=l} |G_{{\beta_i}}|)} $\\  
Hence we get desired result.  
\end{proof}
\begin{Corollary}
Let $D_n=\{f^i r^j|o(f)=2,o(r)=n,rf=fr^{-1}\}$ be a finite non-abelian group with $2n$, then we get following results for degree of various vertices in co-prime order graph of group $D_n$\\
(i) If $o(x) \neq 2$,then difference of degree's of x in Co-prime prime order graph of group $D_n$ and group $\langle r \rangle$ is n.\\ 
(ii) If $o(x)=2$, then $deg(x)=2n-1$.
\end{Corollary}
\begin{proof}
Take $x$ be arbitrary element of $D_n$.\\
Case 1:- If $o(x) \neq 2$, then $x$ is connected with following vertices:- \\
(i) All vertices of the type $fr^j$ where $j=1,2,\ldots,n$. These vertices are $n$ vertices.\\
(ii) Vertices of the type $r^j$ if $(o(x),\frac{n}{(n,j)})=1$ or prime where $j=1,2,\ldots,n$. Number of these vertices are same as degree of any vertex of order equal to o(x) in co-prime order graph of group ${\mathbb{Z}_{n}}$ \\  
Hence, we get difference of degree's of x in Co-prime prime order graph of group $D_n$ and group $\langle r \rangle$ is n.\\ 
Case 2:- If $o(x)=2$, then x is connected with every vertex other than itself, so we get $deg(x)=2n-1$.  
\end{proof}
\section{Laplacian spectrum}
\begin{Theorem}\label{thm1}
Let \[
   L=
  \left( {\begin{array}{cc}
   A & C \\
   C^T & B \\
   
  \end{array} } \right)
\]
where 
\[
   A=
  \left( {\begin{array}{ccccc}
   p+q-1 & -1 & -1 & \ldots & -1\\
   -1 & p+q-1 & -1 & \ldots & -1 \\
   -1 & -1 & p+q-1 & \ldots & -1 \\
   \ldots & \ldots & \ldots &\ldots & \ldots \\
   -1 & -1 & -1 & \ldots & p+q-1\\
  \end{array} } \right)_{p \times p}
\]

$C$ is a $p \times q $ matrix whose all entries are -1 and $B=p I_{q \times q}$. Prove that Laplacian spectrum of $L$ where $p~ and ~q \geq 1$ is $0$ with multiplicity $1,p+q$ with multiplicity $p$ and $p$ with multiplicity $q-1$
\end{Theorem}

\begin{proof}
We now proceed to find the spectrum of L for every prime $p$\\
The characteristic polynomial of L is given by\\
\[ \Lambda=det(xI-L)= det 
   \left( {\begin{array}{cc}
   xI-A & -C \\
   -C^T & xI-B \\
  \end{array} } \right)  
\]
Apply the row operation $R_{1} \to \sum_{i=1}^{p+q} R_i$ and take $x$ common from $1st$ row, we get\\ 
\[ \Lambda=det(xI-L)= x~det 
   \left( {\begin{array}{cc}
   D & F \\
   -C^T & xI-B \\
  \end{array} } \right)  
\]
where 
\[
   D=
  \left( {\begin{array}{cccc}
    1 & 1  & \ldots & 1\\
   1 & x-p-q+1 & \ldots & 1 \\
  
   \ldots & \ldots  &\ldots & \ldots \\
   1 & 1 &  \ldots & x-p-q+1\\
  \end{array} } \right)_{p \times p}
  F=
  \left( {\begin{array}{cccc}
    1 & 1  & \ldots & 1\\
   1 &  1  & \ldots & 1 \\
   
   \ldots & \ldots  &\ldots & \ldots \\
   1 & 1 &  \ldots & 1\\
  \end{array} } \right)_{p \times q}
\]

Now we apply row operations $R_{i} \to R_{i}-R_{1} \forall~i=2,3,\ldots,p+q$, we get\\
\[ \Lambda=det(xI-L)= x~det 
   \left( {\begin{array}{cc}
    G & H \\
    I & J \\
  \end{array} } \right)  
\]
where 
\[
   G=
  \left( {\begin{array}{cccc}
    1 & 1  & \ldots & 1\\
    0 & x-p-q &  \ldots & 0 \\
   \ldots & \ldots  &\ldots & \ldots \\
   0 & 0  & \ldots & p-q\\
  \end{array} } \right)_{p \times p}
  F=
  \left( {\begin{array}{ccccc}
    1 & 1  & \ldots & 1\\
   0 &  0  & \ldots & 0 \\
  
   \ldots &  \ldots &\ldots & \ldots \\
   0 &  0  & \ldots & 0 \\
  \end{array} } \right)_{p \times q}
\]  
\[
  I=
  \left( {\begin{array}{cccc}
   0 &  0  & \ldots & 0 \\
   0  & 0 & \ldots & 0 \\
   \ldots & \ldots  &\ldots & \ldots \\
   0 &  0  & \ldots & 0 \\
  \end{array} } \right)_{q \times p}
  J=
  \left( {\begin{array}{cccc}
    x-p-1 & -1  & \ldots & -1\\
    -1 & x-p-1  & \ldots & -1 \\
   \ldots & \ldots  &\ldots & \ldots \\
   -1 & -1 & \ldots & x-p-1\\
  \end{array} } \right)_{q \times q}
\]

Now we rewrite 
\[ \Lambda=det(xI-L)=x(x-p-q)^{p-1} det(J) \]

Apply the row operation $R_{1} \to \sum_{i=1}^{q} R_i$ to $det(J)$ and take $x-p-q$ common from $1st$ row, we get\\
\[ \Lambda=det(xI-L)=x(x-p-q)^{p} det(K) \]
where 
\[  K=
  \left( {\begin{array}{cccc}
    1 & 1  & \ldots & 1\\
    -1 & x-p-1  & \ldots & -1 \\
   \ldots & \ldots &\ldots & \ldots \\
   -1 & -1 & \ldots & x-p-1\\
  \end{array} } \right)_{q \times q}  
\]
Now we apply row operations $R_{i} \to R_{i}+R_{1} \forall~i=2,3,\ldots,q$, we get\\
\[ \Lambda=det(xI-L)=x(x-p-q)^p det(K_1) \]
where 
\[  K_1=
  \left( {\begin{array}{cccc}
    1 & 1  & \ldots & 1\\
    0 & x-p  & \ldots & 0 \\
   \ldots & \ldots  &\ldots & \ldots \\
   0 & 0 & \ldots & x-p\\
  \end{array} } \right)_{q \times q}  
\]

Finally, we get $\Lambda=det(xI-L)=x(x-p-q)^{p}(x-p)^{q-1}$.\\
Hence the eigenvalues of $L$ are $0$ with multiplicity $1$, $p+q$ with multiplicity $p$ and $p$ with multiplicity $q-1$.
\end{proof}

\begin{Theorem}
Prove that Laplacian spectrum of group ${\mathbb{Z}^{t}_{p}}$ where $p$ is prime and $t \geq 1$ are $0$ with multiplicity $1$ and  $p^t$ with multiplicity $p^t-1$. 
\end{Theorem}
\begin{proof}
We now proceed to find the Laplacian Spectrum of group and denoted by $L$.\\
The rows and columns of the matrix $L$ have been indexed in the following ways:\\
We start with the zero element $[0]$ of ${\mathbb{Z}^{t}_{p}}$.\\
We then list the remaining elements of ${\mathbb{Z}^{t}_{p}}$.\\
Using the above indexing the matrix of $L$ takes the following form:\\
\[
   L=
  \left( {\begin{array}{ccccc}
   p^t-1 & -1 & -1 & \ldots & -1\\
   -1 & p^t-1 & -1 & \ldots & -1 \\
   -1 & -1 & p^t-1 & \ldots & -1 \\
   \ldots & \ldots & \ldots &\ldots & \ldots \\
   -1 & -1 & -1 & \ldots & p^t-1\\
  \end{array} } \right)_{p^t \times p^t}
\]
We now proceed to find the spectrum of $L$ for every prime $p$.The characteristic polynomial of $L$ is given by $\Lambda(x)=det(xI-L)$.\\
Apply the row operation $R_{1} \to \sum_{i=1}^{p^t} R_i$ and take $x$ common from $1st$ row, we get\\  
\[
   \Lambda(x)=x~det
  \left( {\begin{array}{ccccc}
   1 & 1 & 1 & \ldots & 1\\
   1 & x-p^t+1 & 1 & \ldots & 1 \\
   1 & 1 & x-p^t+1 & \ldots & 1 \\
   \ldots & \ldots & \ldots &\ldots & \ldots \\
   1 & 1 & 1 & \ldots & x-p^t+1\\
  \end{array} } \right)_{p^t \times p^t}
\]
Now we apply row operations $R_{i} \to R_{i}-R_{1} \forall~i=2,3,\ldots,p^t$, we get\\
\[
   \Lambda(x)=x~det
  \left( {\begin{array}{ccccc}
   1 & 1 & 1 & \ldots & 1\\
   0 & x-p^t & 0 & \ldots & 0 \\
   0 & 0 & x-p^t & \ldots & 0 \\
   \ldots & \ldots & \ldots &\ldots & \ldots \\
   0 & 0 & 0 & \ldots & x-p^t\\
  \end{array} } \right)_{p^t \times p^t}
\]
Thus  we have, $\Lambda(x)=x(x-p^t)^{p^t-1}$\\
Hence the eigenvalues of L are 0 with multiplicity 1 and $p^t$ with multiplicity $p^t-1$.
\end{proof}

\begin{Theorem}
Prove that Laplacian spectrum of group $\mathbb{Z}_{{p}^{\alpha_1}} \times \mathbb{Z}_{{p}^{\alpha_2}} \times  \ldots \times \mathbb{Z}_{{p}^{\alpha_n}}$ where $n \geq 1$ are $0$ with multiplicity $1$, $p^{\alpha_1+\alpha_2+\ldots+\alpha_n}$ with multiplicity $p^n$ and $p^n$ with multiplicity $p^{\alpha_1+\alpha_2+\ldots+\alpha_n}-p^n-1$ where $p$ is a prime and $ \max{(\alpha_1,\alpha_2, \ldots ,\alpha_n)}\geq 2$
\end{Theorem}

\begin{proof}
We now proceed to find the Laplacian Spectrum of group $\mathbb{Z}_{{p}^{\alpha_1}} \times \mathbb{Z}_{{p}^{\alpha_2}} \times  \ldots \times \mathbb{Z}_{{p}^{\alpha_n}}$ and denoted by $L$.\\
The rows and columns of the matrix $L$ have been indexed in the following ways:\\
We start with the zero element $[0]$ of $\mathbb{Z}_{{p}^{\alpha_1}} \times \mathbb{Z}_{{p}^{\alpha_2}} \times  \ldots \times \mathbb{Z}_{{p}^{\alpha_n}}$.\\
We then list the all $p^n-1$ elements of order $p$ from group $\mathbb{Z}_{{p}^{\alpha_1}} \times \mathbb{Z}_{{p}^{\alpha_2}} \times  \ldots \times \mathbb{Z}_{{p}^{\alpha_n}}$\\
We then list the remaining elements of $\mathbb{Z}_{{p}^{\alpha_1}} \times \mathbb{Z}_{{p}^{\alpha_2}} \times  \ldots \times \mathbb{Z}_{{p}^{\alpha_n}}$\\
Using the above indexing the matrix of $L$ takes the following form:\\
\[
   L=
  \left( {\begin{array}{cc}
   A & C \\
   C^T & B \\
   
  \end{array} } \right)
\]
where 
\[
   A=
  \left( {\begin{array}{ccccc}
   p^{\alpha_1+\alpha_2+\ldots+\alpha_n}-1 & -1 & -1 & \ldots & -1\\
   -1 & p^{\alpha_1+\alpha_2+\ldots+\alpha_n}-1 & -1 & \ldots & -1 \\
   -1 & -1 & p^{\alpha_1+\alpha_2+\ldots+\alpha_n}-1 & \ldots & -1 \\
   \ldots & \ldots & \ldots &\ldots & \ldots \\
   -1 & -1 & -1 & \ldots & p^{\alpha_1+\alpha_2+\ldots+\alpha_n}-1\\
  \end{array} } \right)_{p^n \times p^n}
\]
$C$ is a $p^n \times p^{\alpha_1+\alpha_2+\ldots+\alpha_n}-p^n $ matrix whose all entries are -1 and $B=p I_{p^{\alpha_1+\alpha_2+\ldots+\alpha_n}-p^n}$\\
By using theorem \ref{thm1}, we get Laplacian Spectrum of $L$ as follows:-\\ 
$0$ with multiplicity $1$, $p^{\alpha_1+\alpha_2+\ldots+\alpha_n}$ with multiplicity $p^n$ and $p^n$ with multiplicity $p^{\alpha_1+\alpha_2+\ldots+\alpha_n}-p^n-1$ 
\end{proof}

\begin{Theorem}
Prove that Laplacian spectrum of group ${\mathbb{Z}_p}^t \times {\mathbb{Z}_q}^s$ are $0$ with multiplicity $1$, $p^tq^s$ with multiplicity $p^t+q^s-1$ and $p^t+q^s-1$ with multiplicity $p^tq^s-p^t-q^s$ where $p$ and $q$ are distinct primes and $s,t \geq 1$.
\end{Theorem}

\begin{proof}
We now proceed to find the Laplacian Spectrum of group ${\mathbb{Z}_p}^t \times {\mathbb{Z}_q}^s$ and denoted by $L$.\\
The rows and columns of the matrix $L$ have been indexed in the following ways:\\
We start with the zero element $[0]$ of ${\mathbb{Z}_p}^t \times {\mathbb{Z}_q}^s$\\
We then list the all $p^t-1$ elements of order $p$ and $q^s-1$ elements of order $q$ from group ${\mathbb{Z}_p}^t \times {\mathbb{Z}_q}^s$.\\
We then list the remaining elements of ${\mathbb{Z}_p}^t \times {\mathbb{Z}_q}^s$\\
Using the above indexing the matrix of $L$ takes the following form:\\
\[
   L=
  \left( {\begin{array}{cc}
   A & C \\
   C^T & B \\
   
  \end{array} } \right)
\]
where 
\[
   A=
  \left( {\begin{array}{ccccc}
   p^tq^s-1 & -1 & -1 & \ldots & -1\\
   -1 & p^tq^s-1 & -1 & \ldots & -1 \\
   -1 & -1 & p^tq^s-1 & \ldots & -1 \\
   \ldots & \ldots & \ldots &\ldots & \ldots \\
   -1 & -1 & -1 & \ldots & p^tq^s-1\\
  \end{array} } \right)_{p^t+q^s-1 \times p^t+q^s-1}
\]
$C$ is a $p^t+q^s-1 \times p^tq^s-p^t-q^s+1$ matrix whose all entries are -1 and $B=(p^t+q^s-1) I_{p^tq^s-p^t-q^s+1}$\\
By using theorem \ref{thm1}, we get Laplacian Spectrum of $L$ as follows:-\\ 
$0$ with multiplicity $1$, $p^tq^s$ with multiplicity $p^t+q^s-1$ and $p^t+q^s-1$ with multiplicity $p^tq^s-p^t-q^s$ 
\end{proof}

\begin{Theorem}
Prove that Laplacian spectrum of group $D_{p^n}$ are $0$ with multiplicity $1$, $2p^n$ with multiplicity $p+p^n$ and $p+p^n$ with multiplicity $p^n-p-1$ with condition that if $p$ is a odd prime, then $ n\geq 1$ and if $p$ is even prime than $n \geq 2$.
\end{Theorem}

\begin{proof}
We now proceed to find the Laplacian Spectrum of group $D_{p^n}$ and denoted by $L$.\\
The rows and columns of the matrix $L$ have been indexed in the following ways:\\
We start with the zero element $[0]$ of $D_{p^n}$\\
We then list the all $p-1$ elements of order $p$ and $p^n$ elements of order $2$ from group $D_{p^n}$.\\
We then list the remaining elements of $D_{p^n}$\\
Using the above indexing the matrix of $L$ takes the following form:\\
\[
   L=
  \left( {\begin{array}{cc}
   A & C \\
   C^T & B \\
   
  \end{array} } \right)
\]
where 
\[
   A=
  \left( {\begin{array}{ccccc}
   2p^n-1 & -1 & -1 & \ldots & -1\\
   -1 & 2p^n-1 & -1 & \ldots & -1 \\
   -1 & -1 & 2p^n-1 & \ldots & -1 \\
   \ldots & \ldots & \ldots &\ldots & \ldots \\
   -1 & -1 & -1 & \ldots & 2p^n-1\\
  \end{array} } \right)_{p+p^n \times p+p^n}
\]
$C$ is a $p+p^n \times p^n-p$ matrix whose all entries are -1 and $B=(p+p^n) I_{p^n-p}$\\
By using theorem \ref{thm1}, we get Laplacian Spectrum of $L$ as follows:-\\ 
$0$ with multiplicity $1$, $2p^n$ with multiplicity $p+p^n$ and $p+p^n$ with multiplicity $p^n-p-1$ 
\end{proof}

\bibliographystyle{hplain}

\end{document}